\documentclass{article}

\usepackage{graphicx}
\usepackage{amsmath}
\usepackage{amssymb}
\usepackage{color}
\usepackage{hyperref}
\usepackage{epsfig}
\usepackage{tikz}
\usepackage{pgfplots}
\usepackage{caption}
\usepackage{float}
\allowdisplaybreaks

\newtheorem{theo}{Theorem}

\newtheorem{theorem}{Theorem}

\newtheorem{coro}{Corollary}
\newtheorem{defi}{Definition}

\newtheorem{prop}{Proposition}

\newtheorem{remark}{Remark}

\newcommand{\NN}{{\mathbb N}}

\newcommand{\RR}{{\mathbb R}}

\newenvironment{proof}{\begin{trivlist}
   \item[\hskip\labelsep{\it Proof}]}{$\hfill\Box$\end{trivlist}}

\newcommand{\notiz}[1]{}

\title{Some negative results related to Poissonian pair correlation problems}
\author{Gerhard Larcher \footnote{The author is supported by the Austrian Science Fund (FWF), Project F5507-N26, which is a part of the Special Research Program “Quasi-Monte Carlo Methods: Theory and Applications” and Project I1751-N26.} \ , Wolfgang Stockinger \footnote{The author is supported by the Austrian Science Fund (FWF), Project F5507-N26, which is a part of the Special Research Program “Quasi-Monte Carlo Methods: Theory and Applications”.}}
\date{}
\begin{document}
\maketitle 
\begin{abstract}
We say that a sequence $(x_n)_{n \in \NN}$ in $[0,1)$ has Poissonian pair correlations if 
\begin{equation*}
\lim_{N \to \infty} \frac{1}{N} \# \left \lbrace 1  \leq l \neq m \leq N: \| x_l - x_m \| \leq \frac{s}{N} \right \rbrace = 2s
\end{equation*} 
for every $s \geq 0$. The aim of this article is twofold. First, we will establish a gap theorem which allows to deduce that a sequence $(x_n)_{n \in \NN}$ of real numbers in $[0,1)$ having a certain weak gap structure, cannot have Poissonian pair correlations. This result covers a broad class of sequences, e.g., Kronecker sequences, the van der Corput sequence and in more general $LS$-sequences of points and digital $(t,1)$-sequences. Additionally, this theorem enables us to derive negative pair correlation properties for sequences of the form $(\lbrace a_n \alpha \rbrace)_{n \in \mathbb{N}}$, where $(a_n)_{n \in \NN}$ is a strictly increasing sequence of integers with maximal order of additive energy, a notion that plays an important role in many fields, e.g., additive combinatorics, and is strongly connected to Poissonian pair correlation problems. These statements are not only metrical results, but hold for all possible choices of $\alpha$.

Second, in this note we study the pair correlation statistics for sequences of the form, $x_n = \lbrace b^n \alpha \rbrace, \ n=1, 2, 3, \ldots$, with an integer $b \geq 2$, where we choose $\alpha$ as the Stoneham number and as an infinite de Bruijn word. We will prove that both instances fail to have the Poissonian property. Throughout this article $\lbrace \cdot \rbrace$ denotes the fractional part of a real number. 
\end{abstract} 
\section{Introduction and statement of main results}
The concept of Poissonian pair correlations has its origin in quantum mechanics, where the spacings of energy levels of integrable systems were studied. See for example \cite{not9} and the references cited therein for detailed information on that topic. Rudnik and Sarnak first studied this concept from a purely mathematical point of view and over the years the topic has attracted wide attention, see e.g., \cite{ not5, not10, not1, not2, not3}. \\

Let $\| \cdot \|$ denote the distance to the nearest integer. A sequence $(x_n)_{n \in \NN}$ of real numbers in $[0,1)$ has Poissonian pair correlations if the pair correlation statistics
\begin{equation}\label{eq:pc1}
F_N(s):= \frac{1}{N} \# \left\lbrace 1  \leq l \neq m \leq N: \| x_l - x_m \| \leq \frac{s}{N} \right\rbrace 
\end{equation} 
tends to $2s$, for every $s \geq 0$, as $N \to \infty$. \\

Recently, Aistleitner, Larcher, Lewko and Bourgain (see \cite{not6}) could give a strong link between the concept of Poissonian pair correlations and the additive energy of a finite set of integers, a notion that plays an important role in many mathematical fields, e.g., in additive combinatorics. To be precise, for a finite set $A$ of reals the additive energy $E(A)$ is defined as  
\begin{equation*}
E(A):= \sum_{a+b=c+d} 1,
\end{equation*}
where the sum is extended over all quadruples $(a,b,c,d) \in A^4$. Roughly speaking, it was proved in \cite{not6} that if the first $N$ elements of an increasing sequence of distinct integers $(a_n)_{n \in \NN}$, have an arbitrarily small energy saving, then $(\lbrace a_n \alpha \rbrace)_{n \in \NN}$ has Poissonian pair correlations for almost all $\alpha$. In this paper the authors also raised the question if $(\lbrace a_n \alpha \rbrace)_{n \in \NN}$, where $(a_n)$ is an increasing sequence of distinct integers with maximal order of additive energy, can have Poissonian pair correlations for almost all $\alpha$. Jean Bourgain could show that the answer to this question is negative, i.e., he proved:
\begin{theo}[in \cite{not6}]
If $E(A_N) = \Omega(N^3)$, where $A_N$ denotes the first $N$ elements of $(a_n)_{n \in \NN}$, then there exists a subset of $[0,1]$ of positive measure such that for every $\alpha$ from this set the pair correlations of $(\lbrace a_n \alpha \rbrace)_{n \in \NN}$ are not Poissonian.
\end{theo}
Recently, the result of Bourgain has been further extended, see \cite{not9, not14, not15, not31}. The result given in \cite{not15} is an easy consequence of our Theorem 1 stated below and will be shown in Section 2. Further, see \cite{not33, not32} for (negative) results and discussions concerning a Khintchine type criterion which fully characterizes the metric pair correlation property in terms of the additive energy. \\

Due to a result by Grepstad and Larcher \cite{not8} (see also \cite{not7, not12}), we know that a sequence which satisfies that (\ref{eq:pc1}) tends to $2s$, for every $s \geq 0$, as $N \to \infty$, is also uniformly distributed in $[0,1)$, i.e., it satisfies 
\begin{equation*}
\lim_{N \to \infty} \frac{1}{N} \# \lbrace 1 \leq n \leq N: x_n \in [a,b) \rbrace = b-a 
\end{equation*}
for all $0 \leq a < b \leq 1$. Note that the other direction is not necessarily correct. For instance the Kronecker sequence $(\lbrace n\alpha \rbrace)_{n \in \NN}$ is uniformly distributed modulo $1$ for irrational $\alpha$, but does not have Poissonian pair correlations for any real $\alpha$; a fact that easily follows from continued fractions arguments. In earlier papers (see e.g., \cite{not6, not3}) this fact was argued to be an immediate consequence of the Three Gap Theorem \cite{not4}. The Three Gap Theorem, roughly speaking, states that the Kronecker sequence always has at most three distinct distances between nearest sequence elements. Nonetheless -- at least for us -- it is not immediately clear that we can deduce from this fact that $(\lbrace n\alpha \rbrace)_{n \in \NN}$ is not Poissonian for any $\alpha$. Therefore, we will prove the following very general result concerning the link between Poissonian pair correlations and a certain gap structure of a sequence in the unit interval. In the next section, we will present some applications of this Theorem 1. 
\begin{theorem}
Let $(x_n)_{n \in \NN}$ be a sequence in $[0,1)$ with the following property: There is an $s \in \NN$, positive real numbers $K$ and $\gamma$, and infinitely many $N$ such that the point set $x_1, \ldots, x_N$ has a subset with $M \geq \gamma N$ elements, denoted by $x_{j_1}, \ldots, x_{j_M}$, which are contained in a set of points with cardinality at most $KN$ having at most $s$ different distances between neighbouring sequence elements, so-called gaps. Then, $(x_n)_{n \in \NN}$ does not have Poissonian pair correlations.
\end{theorem}
Poissonian pair correlation is a typical property of a sequence. Random sequences, i.e., almost all sequences, have the Poissonian pair correlation property.
Nevertheless, it seems to be extremely difficult to give explicit examples of sequences with Poissonian pair correlations. We note that $(\lbrace \sqrt{n} \rbrace)_{n \in \NN}$ has Poissonian pair correlations, \cite{not27} (see \cite{not28} for another explicit construction). Apart from that -- to our best knowledge -- no other explicit examples are known. Especially, until now we do not know any single explicit construction of a real number $\alpha$ such that the sequence of the form $(\lbrace a_n \alpha \rbrace)_{n \in \NN}$ has Poissonian pair correlations.  \\

We recall that the sequence $(\lbrace b^n \alpha \rbrace)_{n \in \NN}$, for an integer $b \geq 2$, has the Poissonian property for almost all $\alpha$. Moreover we know that the sequence  $(\lbrace b^n \alpha \rbrace)_{n \in \NN}$ is uniformly distributed modulo $1$ if and only if $\alpha$ is normal in base $b$, see e.g., \cite{not11}. If we want to investigate, whether the distribution of the pair correlations for some explicit given sequence is Poissonian, the sequence has to be uniformly distributed modulo $1$. Therefore, if we study the distribution of the spacings between the sequence elements of $(\lbrace b^n \alpha \rbrace)_{n \in \NN}$, the only reasonable choice for $\alpha$ is a $b$-normal number. In $\cite{not22}$, the sequence $(\lbrace 2^n \alpha \rbrace)_{n \in \NN}$ was studied for the Champernowne constant $\alpha$ and it was shown that it is not Poissonian. In this note we will choose two other special instances, which were suggested by Yann Bugeaud as potential candidates in personal communication. First, we will consider so-called infinite de Bruijn words.
\begin{defi}
A (non cyclic) de Bruijn word of order $m$ over an alphabet $A$ is a word of length $|A|^m+m-1$ such that every word of length $m$ occurs in it exactly once. 
\end{defi}
\begin{defi}
An infinite de Bruijn word $w=a_1a_2 \ldots$ in an alphabet of at least three symbols is an infinite word such that, for every $m$, $a_1 \ldots a_{|A|^m+m-1}$ is a de Bruijn word of order $m$. In case the alphabet has two symbols, an infinite de Bruijn word $w= a_1 a_2 \ldots $ is such that, for every odd $m$, $a_1 \ldots a_{|A|^m+m-1}$ is a de Bruijn word of order $m$. 
\end{defi}
It is known that infinite de Bruijn words are normal, see, e.g., \cite{not19}.
We will prove the following theorem. 
\begin{theorem}
The sequence $( \lbrace b^n \alpha \rbrace)_{n \in \NN}$, $b \geq 2$, where $\alpha=0.a_1a_2 \ldots $ and $a_1a_2 \ldots$ is an infinite de Bruijn word, does not have Poissonian pair correlations. 
\end{theorem}
For further properties of de Bruijn words, we refer the reader to \cite{not21, not20}. \\

As a second instance, we study the Stoneham number 
\begin{equation*}
\alpha_{2,3} := \sum_{m=1}^{\infty} \frac{1}{3^m 2^{3^m}},
\end{equation*}
which is known to be $2$-normal, see, e.g., \cite{not17}.
Again, we obtain a negative result. 
\begin{theorem}
The sequence $(\lbrace 2^n \alpha_{2,3} \rbrace)_{n \in \NN}$ does not have Poissonian pair correlations.  
\end{theorem}
\section{Applications of Theorem 1}
Theorem $1$ immediately allows us to deduce that Kronecker sequences, due to the Three Gap Theorem, the classical van der Corput sequence and in more general, sequences which have a not too small intersection with these sequences, do not have Poissonian pair correlations. In the following, we give three further examples of classes of sequences for which Theorem 1 can be applied. \\ \\
\textbf{Example 1: $LS$-sequences}  \\ \\
Our presentation follows the steps of \cite{not26} (see also, \cite{not24, not25}).
\begin{defi}
Let $\rho$ denote a non-trivial partition of $[0,1)$. Then the $\rho$-refinement of a partition $\pi$ of $[0,1)$, denoted by $\rho \pi$, is defined by subdividing all intervals of maximal length positively homothetically to $\rho$.
\end{defi} 
This partitioning procedure results in a sequence of partitions abbreviated by $(\rho^n \pi)_{n \in \NN}$. Now, we are in the position to define so-called $LS$-sequences of partitions and $LS$-sequences of points, introduced in \cite{not23}.
\begin{defi}
Let $L \in \NN$, $S \in \NN_{0}$, $L+S \geq2$ and $\beta$ be the solution of $L\beta + S\beta^2 =1$. An $LS$-sequence of partitions $(\rho_{L,S}^n \pi)_{n \in \NN}$ is the successive $\rho$-refinement of the trivial partition $\pi = \lbrace [0,1) \rbrace$ where $\rho_{L,S}$ consists of $L+S$ intervals such that the first $L$ intervals have length $\beta$ and the successive $S$ intervals have length $\beta^2$.
\end{defi}  
\begin{defi}
Given an $LS$-sequence of partitions $(\rho_{L,S}^n \pi)_{n \in \NN}$, the corresponding $LS$-sequence of points $(\xi_{L,S}^n)_{n \in \NN}$ is defined as follows: let $\Lambda_{L,S}^1$ be the first $t_1$ left endpoints of the partition $\rho_{L,S} \pi$ ordered by magnitude. Given $\Lambda_{L,S}^n = \lbrace \xi_{L,S}^1, \ldots, \xi_{L,S}^{t_n}  \rbrace$ an ordering of $\Lambda_{L,S}^{n+1}$ is inductively defined as 
\begin{align*}
\Lambda_{L,S}^{n+1} = \lbrace &  \xi_{L,S}^1, \ldots, \xi_{L,S}^{t_n},  \\
& \psi_{1,0}^{n+1}(\xi_{L,S}^1), \ldots, \psi_{1,0}^{n+1}(\xi_{L,S}^{l_n}), \ldots, \psi_{L,0}^{n+1}(\xi_{L,S}^1), \ldots, \psi_{L,0}^{n+1}(\xi_{L,S}^{l_n}), \\
&  \psi_{L,1}^{n+1}(\xi_{L,S}^1), \ldots, \psi_{L,1}^{n+1}(\xi_{L,S}^{l_n}), \ldots, \psi_{L,S-1}^{n+1}(\xi_{L,S}^1), \ldots, \psi_{L,S-1}^{n+1}(\xi_{L,S}^{l_n}) \rbrace, 
\end{align*}
where $\psi_{i,j}^n = x + i \beta^n + j \beta^{n+1}$, \quad $x \in \RR$. 
\end{defi}
Due to the definition of a $LS$-sequence of partitions, we see that there are only two distinct gap lengths which are of the form $\beta^n$ and $\beta^{n+1}$, $n \in \NN$. Therefore, due to Theorem 1, we obtain the following corollaries.
\begin{coro} 
$LS$-sequences of points do not have Poissonian pair correlations.
\end{coro} 
\begin{coro}
Let $(x_n)_{n \in \NN}$ be a sequence in $[0,1)$ with the following property: There is a constant $\kappa>0$, a sequence $N_1 < N_2, \ldots, $ of positive integers and for each $N_i$, $i \geq 1$, a $LS$-sequence $(y_n^{(i)})_{n \in \NN}$ such that
\begin{equation*}
| \lbrace x_1, \ldots, x_{N_i} \rbrace \cap \lbrace y_1^{(i)}, \ldots, y_{N_i}^{(i)} \rbrace  | \geq \kappa N_i,
\end{equation*}
then $(x_n)_{n \in \NN}$ does not have Poissonian pair correlations.  
\end{coro}
Note that for $S=0$ and $L=b$, we get the classical van der Corput sequence. \\ \\
\textbf{Example 2: Quasi-arithmetic sequences of degree $1$}\\ \\
As a second application of Theorem 2, we illustrate that we can recover a recent result by Larcher \cite{not15} which extends Theorem A of Bourgain, mentioned in the introduction, for a special class of sequences. First, we need the definition of so-called quasi-arithmetic sequences of degree $d$, see \cite{not9, not15}. 
\begin{defi}
Let $(a_n)_{n \in \NN}$ be a strictly increasing sequence of positive integers. We call this sequence quasi-arithmetic of degree $d$, where $d$ is a positive integer, if there exist constants $C,K >0$ and a strictly increasing sequence $(N_i)_{i \in \NN}$ of positive integers such that for all $i \geq 1$ there is a subset $A^{(i)} \subset (a_n)_{1 \leq n \leq N_i}$ with $| A^{(i)}| \geq C N_i$ such that $A^{(i)}$ is contained in a $d$-dimensional arithmetic progression $P^{(i)}$ of size at most $KN_i$. 
\end{defi} 
Further, it is known (see \cite{not9} for a proof) that a strictly increasing sequence $(a_n)_{n \in \NN}$ of positive integers is quasi-arithmetic of some degree $d$ if and only if $E(A_N) = \Omega(N^3)$. Therefore, studying the pair correlations of sequences with maximal order of additive energy amounts to investigating quasi-arithmetic sequences of some degree. Our Theorem 2 allows to recover the following result.  
\begin{theo}[Theorem 1 in \cite{not15}]
If $(a_n)_{n \in \NN}$ is quasi-arithmetic of degree $d=1$, then there is no $\alpha$ such that the pair correlations of $(\lbrace a_n \alpha \rbrace)_{n \in \NN}$ are Poissonian. 
\end{theo} 
To argue this, we note that for infinitely many $N$ the set $(\lbrace a_n \alpha \rbrace)_{n=1, \ldots, N}$, where $(a_n)_{n \in \NN}$ is quasi-arithmetic of degree $d=1$, contains a subset of the form $S_{\alpha}:= (\lbrace a_{n_j} \alpha \rbrace)_{j =1, \ldots, L}$, $L \geq c N$, for some constant $c>0$, where we have $S_{\alpha} \subseteq (\lbrace j \alpha \rbrace)_{j=1, \ldots, M}$, $M \leq KN$, for some constant $K>0$. This means, for infinitely many $N$, there are subsets of $(\lbrace a_n \alpha \rbrace)_{n=1, \ldots, N}$ which are contained in point sets having at most three distinct gaps. \\
    
We mention that due to a result of Lachmann and Technau (\cite{not14}), we can immediately deduce that for almost all $\alpha$ the pair correlations of $(\lbrace a_n \alpha \rbrace)_{n \in \NN}$ are not Poissonian, if $(a_n)_{n \in \NN}$ is a quasi-arithmetic sequence of degree $d \geq 2$. \\

Recently, this result was further improved by the authors, who showed that there is \textbf{no} $\alpha$ such that the pair correlations of $(\lbrace a_n \alpha \rbrace)_{n \in \NN}$ are Poissonian, if $(a_n)_{n \in \NN}$ is quasi-arithmetic of degree $d$, for $d \geq 1$ (\cite{not31}). \\ \\
\textbf{Example 3: $(t,1)$-sequences} \\ \\
We will illustrate based on our results presented in the previous section that digital $(t,1)$-sequences, introduced by H.\ Niederreiter \cite{not30}, cannot have Poissonian pair correlations. For the definition and properties of $(t,s)$-sequences, we refer to \cite{not29}. We obtain the following corollaries:  
\begin{coro} 
A digital $(t,1)$-sequence has the finite gap property for all $N \in \NN$. Hence, digital $(t,1)$-sequences do not have Poissonian pair correlations.
\end{coro} 
\begin{proof}
Consider first a digital $(0,1)$-sequence $(y_n)_{n \in \NN_0}$ in base $b \geq 2$. By definition of a $(0,1)$-sequence, we have that for any $m \in \NN$ and any $k \in \NN_0$ the point set 
\begin{equation*}
\mathcal{P}:=\lbrace y_{kb^m}, \ldots, y_{kb^m+ b^m-1} \rbrace
\end{equation*}
forms a $(0,m,1)$-net in base $b$. I.e., we know that each elementary interval of length $1/b^m$ contains exactly one point of $\mathcal{P}$. Due to the digital method, i.e., the construction method of digital $(t,s)$-sequences, the elements of $\mathcal{P}$ have the form
\begin{equation*}
y_n = \frac{y_{n,1}}{b} + \frac{y_{n,2}}{b^2} + \ldots + \frac{y_{n,m}}{b^m}, 
\end{equation*}
where $y_{n,1}, y_{n,2} \ldots, y_{n,m} \in \lbrace 0, 1, \ldots, b-1 \rbrace$ and $n \in \lbrace kb^m, \ldots, kb^m+ b^m-1 \rbrace $. This is the result, when multiplying the generator matrix  with a vector containing the coefficients of the $b$-adic digit expansion of the integer $n$. Therefore, we can conclude that the distance between two neighbouring points of $\mathcal{P}$ is $1/b^m$. For a quality parameter $t \geq 1$, every elementary interval of length $1/b^{m-t}$ contains exactly $b^t$ point. In this case, due to the linear dependence of $t+1$ rows of the generator matrix, our point set $\mathcal{P}_t$ contains only $b^{m-t}$ distinct points and each of those coincides with $b^t-1$ other points of $\mathcal{P}_t$. Therefore, for quality parameters $t \geq 1$, the distance between two neighbouring points of $\mathcal{P}_t$ is $1/b^{m-t}$. For $b^{m-1} < N < b^{m}$, we note that there are three distinct gaps between neighbouring elements of the first $N$ elements of some $(t,1)$-sequence. I.e., we have the finite gap property for all $N \in \NN$.
\end{proof}
\begin{coro}
Let $(x_n)_{n \in \NN}$ be a sequence in $[0,1)$ with the following property: There is a constant $\kappa>0$, a sequence $N_1 < N_2, \ldots, $ of positive integers and for each $N_i$, $i \geq 1$, a digital $(t,1)$-sequence $(y_n^{(i)})_{n \in \NN}$ such that
\begin{equation*}
| \lbrace x_1, \ldots, x_{N_i} \rbrace \cap \lbrace y_1^{(i)}, \ldots, y_{N_i}^{(i)} \rbrace  | \geq \kappa N_i,
\end{equation*}
then $(x_n)_{n \in \NN}$ does not have Poissonian pair correlations.  
\end{coro}
\begin{remark}
We note that also \textbf{general} (not necessarily digital) $(0,1)$-sequences in base $b$ do not have Poissonian pair correlations. To see this, we consider the sequence of positive integers $N_1 = b^{m_1} < N_2 = b^{m_2} < \ldots$. We now aim at counting the relative number of pairs of points having distance $\leq 1/N_i$, for $i=1, 2, \ldots$. As every elementary interval of length $1/b^{m_i} = 1/N_i$ contains exactly one point, two non-neighbouring points have a distance $>1/N_i$. Hence, we have 
\begin{equation*}
\frac{1}{N_i} \# \left\lbrace 1 \leq l \neq k \leq N_i \  | \ \| x_l -x_k \| \leq \frac{1}{N_i} \right\rbrace \leq 2. 
\end{equation*} 
If $N_i-o(N_i)$ neighbouring points have an exact distance of $1/N_i$ for infinitely many $i$, we attain equality in the previous expression, as $i \to \infty$. In this case though, we can use Theorem 1 to conclude that a $(0,1)$-sequence is not Poissonian. \\

Another possibility to achieve equality in the pair correlation statistics is that $N_i-o(N_i)$ neighbouring points have a distance $ \leq 1/N_i$. In this case, we can choose $s=1/2$ in the definition of the pair correlation function. I.e., in order to obtain 
\begin{equation*}
\lim_{i \to \infty} \frac{1}{N_i} \# \left\lbrace 1 \leq l \neq k \leq N_i \ | \ \| x_l -x_k \| \leq \frac{1}{2N_i} \right\rbrace = 1,
\end{equation*}
$N_i-o(N_i)$ ordered pairs of points need to have a distance smaller than $\frac{1}{2N_i}$. Thus, the length of (about) every second gap has to be $\leq \frac{1}{2N_i}$. Consequently, the lengths of the remaining gaps need to be $\geq 1/N_i$. This, however, means that about $N_i$ ordered pairs of points have a distance $>1/N_i$ (if $\mathcal{O}(N_i)$ pairs have an exact distance of $1/N_i$, we can again apply Theorem 1) and we get a contradiction if we choose $s=1$ in the pair correlation function.  \\
We also strongly believe that \textbf{general} $(t,1)$-sequences in base $b$, for a non-zero quality parameter $t$, fail to have Poissonian pair correlations.
\end{remark}
\section{Proof of Theorem 1}
To us, it seems to be helpful to divide the proof of Theorem 1 into two steps. In the first step, we prove a weaker result which we formulate as Proposition 1 below. In the second step, we prove the Theorem. 
\begin{prop}
Let $(x_n)_{n \in \NN}$ be a sequence in $[0,1)$ with the following property: There is an $s \in \NN$, a positive real number $\gamma$, and infinitely many $N$ such that the point set $x_1, \ldots, x_N$ has a subset with $M \geq \gamma N$ elements, denoted by $x_{j_1}, \ldots, x_{j_M}$, which has at most $s$ different distances between neighbouring sequence elements. Then, $(x_n)_{n \in \NN}$ does not have Poissonian pair correlations.
\end{prop}
\textit{Proof of Proposition 1} \\ \\
Let $s$ be minimal with the property formulated in Proposition 1, then we can choose $N_1 < N_2 < \ldots$ such that for all $i$ the set of points $x_1, \ldots, x_{N_i}$ has a subset of points $x_{j_1^{(i)}}, \ldots, x_{j_{M_i}^{(i)}}$, $M_i \geq \gamma N_i$, having exactly $s$ gaps. In the sequel we will write $j_k:= j_k^{(i)}$, for all $k=1, \ldots, M_i$. For every $i$, we denote the lengths of these gaps by $d_1^{(i)} < d_2^{(i)} < \ldots < d_s^{(i)}$. Let $w_1$, $0 \leq w_1 <s$, be maximal such that the following holds: For all $\epsilon >0$ there are infinitely many $i$ such that $d_{w_1}^{(i)} \leq \epsilon/M_i \leq \epsilon / \gamma N_i$. Then, we can choose the sequence $(N_i)_{i \in \NN}$ or $(M_i)_{i \in \NN}$ such that $\lim_{i \to \infty} N_i d_{w_1}^{(i)} = \lim_{i \to \infty} M_i d_{w_1}^{(i)} =0$. Therefore, there exists an $L_1>0$ such that $d_{w_1 +1}^{(i)} \geq L_1/M_i \geq L_1 / N_i$ for all $i$ large enough (w.\ l.\ o.\ g.\ we may choose $(N_i)$ or $(M_i)$ such that this holds for all $i$). \\

Now, let $w_2$, $w_1 < w_2 \leq s$, be minimal with the following property: For all $\epsilon > 0$ there are infinitely many $i$ with $d_{w_2}^{(i)} \geq \epsilon / M_i \geq \epsilon /N_i$ and consequently we can choose w.\ l.\ o.\ g.\ $(N_i)$ or $(M_i)$ such that $\lim_{i \to \infty} N_i d_{w_2}^{(i)} = \lim_{i \to \infty} M_i d_{w_2}^{(i)} = \infty$. Hence, there is an $L_2 >0$ such that $d_{w_2 -1}^{(i)} \leq L_2/M_i \leq L_2 / \gamma N_i$ for all $i$ large enough (w.\ l.\ o.\ g.\ we may choose $(N_i)$ or $(M_i)$ such that this holds for all $i$).    \\

To sum up, we can choose $(N_i), (M_i), w_1, w_2, L_1, L_2$ such that 
\begin{equation*}
\lim_{i \to \infty} N_i d_{j}^{(i)} = 0 \text { for } j=1, \ldots, w_1 \text{ and } \lim_{i \to \infty} N_i d_{j}^{(i)} = \infty  \text{ for } j=w_2, \ldots, s
\end{equation*}
and 
\begin{equation*}
L_1/N_i \leq L_1/M_i \leq d_{j}^{(i)} \leq  L_2/M_i \leq L_2/\gamma N_i, \text{ for } j=w_1+1, \ldots, w_2 -1. 
\end{equation*}
Above equations also hold if $N_i$ is replaced by $M_i$. \\

Let $l_j^{(i)}$ denote the number of gaps of length $d_j^{(i)}$. Clearly, $\lim_{i \to \infty} l_j^{(i)}/M_i = 0$, for all $j= w_2, \ldots, s$. \\ \\
\textit{Proof Part 1:} \\

Let now $j$ be such that $w_1 +1 \leq j \leq w_2 -1$. Assume that $l_j^{(i)}/M_i$ does not tend to zero, i.e., there is a $\delta >0$ such that $l_j^{(i)} > \delta M_i \geq \delta \gamma N_i$ for all $i$, i.e., also $l_j^{(i)}/N_i$ does not tend to zero. We will show that in this case $(x_n)_{n \in \NN}$ cannot have Poissonian pair correlations. We partition the interval $[L_1, L_2/ \gamma]$ into $\lceil 2 \frac{L_2/ \gamma -L_1}{\gamma \delta} \rceil$ parts of equal length, i.e., the maximal length is at most $\gamma\delta/2$. There is one such subinterval $[s_1,s_2/ \gamma]$ with
\begin{equation*}
\frac{s_1 }{N_i} \leq \frac{s_1}{M_i} \leq d_j^{(i)} \leq \frac{s_2}{M_i} \leq \frac{s_2}{\gamma N_i},
\end{equation*} 
for infinitely many $i$. Note that there are at least $2 l_j^{(i)}$ pairs of elements of $x_{j_1}, \ldots, x_{j_{M_i}}$ having distance $d_j^{(i)}$. Hence,
\begin{align*}
& \# \lbrace l \neq m \in \lbrace j_1, \ldots, j_{M_i} \rbrace: \| x_l -x_m \| \leq s_2 / \gamma N_i \rbrace \\
& \geq \# \lbrace l \neq m \in \lbrace j_1, \ldots, j_{M_i} \rbrace: \| x_l -x_m \| \leq s_2/M_i \rbrace \\
& \geq \# \lbrace l \neq m \in \lbrace j_1, \ldots, j_{M_i} \rbrace: \| x_l -x_m \| \leq s_1/M_i \rbrace + 2 l_j^{(i)} \\
& \geq \# \lbrace l \neq m \in \lbrace j_1, \ldots, j_{M_i} \rbrace: \| x_l -x_m \| \leq s_1/N_i \rbrace + 2 l_j^{(i)},
\end{align*}
and thus also $N_i F_{N_i}(s_2 / \gamma) \geq N_i F_{N_i}(s_1) +2 l_j^{(i)}$. \\

If the pair correlations of $(x_n)_{n \in \NN}$ were Poissonian, then we had 
\begin{equation*}
\frac{2 s_2}{\gamma} - 2s_1 = \lim_{i \to \infty} F_{N_i}(s_2 / \gamma) - \lim_{i \to \infty} F_{N_i}(s_1) \geq \limsup_{i \to \infty} \frac{2 l_j^{(i)}}{N_i} \geq 2 \gamma \delta, 
\end{equation*}
which is a contradiction as we have $s_2/ \gamma -s_1 < \gamma \delta$. \\ \\
\textit{Proof Part 2:} \\

We can deduce from Poissonianess that $\lim_{i \to \infty} l_j^{(i)}/N_i = \lim_{i \to \infty} l_j^{(i)}/M_i = 0$, for all $j= w_1 +1, \ldots, s$ and therefore 
\begin{equation*}
\lim_{i \to \infty} \frac{l_1^{(i)} + \ldots + l_{w_1}^{(i)}}{M_i} = 1. 
\end{equation*}
We define $\tilde{\tilde{l}}:= l^{(i)}:= l_{w_1+1}^{(i)} + \ldots + l_{s}^{(i)}$ and choose $i$ large enough, such that $ d_{w_1}^{(i)} \leq \frac{1}{2 M_i}$; a requirement that will be needed at a later step of the proof. Also note that $\tilde{\tilde{l}} = o(M_i)$. \\

In the sequel, we will call distances $d_j^{(i)}$ for $j = w_1 + 1, \ldots, s$ "large gaps" and the remaining ones "small gaps". We will divide the unit interval in $\tilde{l} \leq 2 \tilde{\tilde{l}} +1$ subintervals in the following manner: The largest possible union of neighbouring large gaps form an open interval and the largest possible union of neighbouring small gaps form an closed interval, as illustrated in Figure 1. 
\begin{figure}[H]
	\centering
  \includegraphics[width=0.8\textwidth]{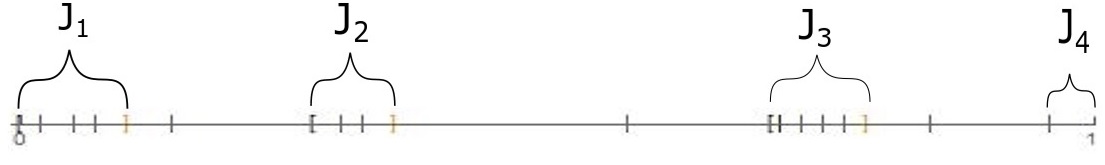}
	\caption{The splitting of the unit interval into unions of small gaps and large gaps. }
	\label{fig1}
\end{figure}  
We denote the $l \leq \tilde{\tilde{l}} +1$ closed intervals formed by the small gaps by $\mathcal{J}_1, \ldots, \mathcal{J}_l$ and their lengths by $J_1, \ldots, J_l$. We partition each such interval $\mathcal{J}_j$ into $\lfloor J_j M_i \rfloor$ intervals of length $1/M_i$. The intervals obtained by this splitting are denoted by $\mathcal{K}_1, \ldots, \mathcal{K}_n$. The number of points of $x_{j_1}, \ldots, x_{j_{M_i}}$, lying in $\mathcal{K}_j$ is denoted by $v_j$. As we have chosen $i$ large enough, such that $ d_{w_1}^{(i)} \leq \frac{1}{2 M_i}$, most intervals $\mathcal{K}_j$ contain at least two points of $x_{j_1}, \ldots, x_{j_{M_i}}$ which have a distance $\leq 1/M_i$ and only $l = o(M_i)$ intervals contain only one point. Hence, for such intervals $\mathcal{K}_j$ containing at least two point, we get  
\begin{equation*}
\# \lbrace l \neq m \in \lbrace j_1, \ldots, j_{M_i} \rbrace: x_l, x_m \in \mathcal{K}_j \text{ and } \| x_l - x_m \| \leq \frac{1}{M_i} \rbrace \geq \frac{v_j^2}{2}
\end{equation*}
and consequently 
\begin{equation*}
\# \lbrace l \neq m \in \lbrace j_1, \ldots, j_{M_i} \rbrace: \| x_l - x_m \| \leq \frac{1}{M_i} \rbrace \geq \sum_{j=1}^n\!{\vphantom{\sum}}^* \frac{v_j^2}{2},
\end{equation*}
where the summation $\sum\!{\vphantom{\sum}}^*$ means that only intervals $\mathcal{K}_j$ containing at least two points of $x_{j_1}, \ldots, x_{j_{M_i}}$ are taken into account.   Recall that both $\tilde{\tilde{l}}$ and $l$ are of the form $o(M_i)$; i.e., all of the $M_i$ points, except $o(M_i)$ many, are contained in the intervals $\mathcal{K}_j$). Therefore we have 
\begin{align*}
& \sum_{j=1}^n\!{\vphantom{\sum}}^* v_j \geq  M_i - o(M_i) \geq \gamma N_i - o(N_i), \\
& n \leq M_i (J_1 + \ldots + J_l) + l \leq N_i (J_1 + \ldots + J_l) + l, \\ 
\end{align*}
and consequently we obtain employing the Cauchy-Schwarz inequality and the fact that $J_1 + \ldots + J_l \leq M_i d_{w_1}^{(i)} \to 0$, as $i \to \infty$,
\begin{align*}
\frac{1}{\gamma} F_{N_i}(1/ \gamma) & \geq \frac{1}{M_i} \# \lbrace 1 \leq l \neq m \leq N_i : \| x_l - x_m \| \leq \frac{1}{\gamma N_i}  \rbrace \\
&\geq \frac{1}{M_i} \# \lbrace l \neq m \in \lbrace j_1, \ldots, j_{M_i} \rbrace: \| x_l - x_m \| \leq \frac{1}{M_i} \rbrace \\
&\geq \frac{1}{M_i} \sum_{j=1}^n\!{\vphantom{\sum}}^* \frac{v_j^2}{2} \geq \frac{n}{M_i} \left( \frac{M_i - o(M_i)}{n} \right)^2 \\
&\geq  \frac{1}{M_i (J_1 + \ldots + J_l) + l} \frac{(M_i - o(M_i))^2}{M_i} \to \infty, \qquad i \to \infty.
\end{align*}
This allows us to deduce that $(x_n)_{n \in \NN}$ cannot have Poissonian pair correlations, and proves Proposition 1. \hfill $\square$ \\ \\
\textit{Proof of Theorem 1} \\ \\
The set of points $x_{j_1}, \ldots, x_{j_M}$ (recall that $M \geq \gamma N$) is contained in the set $\lbrace z_1, \ldots, z_{KN} \rbrace$, where this set of points has exactly $s$ distinct gaps. In the sequel, we put $K=1$. The general case is treated quite similarly. We again denote the gaps by $d_1 < \ldots < d_s$. Hence, the gaps between the points $x_{j_1}, \ldots, x_{j_M}$ have the form 
\begin{equation*}
K_1 d_1 + K_2 d_2+ \ldots + K_s d_s, \text{ where } K_1, K_2, \ldots, K_s \in \NN_0.
\end{equation*}
The number of gaps of the form $K_1 d_1 + K_2 d_2+ \ldots + K_s d_s$ is denoted by $l_{K_1, K_2, \ldots, K_s}$. Note that we have 
\begin{equation}\label{eq:eq11}
\sum_{K_1, K_2, \ldots, K_s \geq 0} l_{K_1, K_2, \ldots, K_s} \geq \gamma N -1.
\end{equation}
Further, the following equality holds:
\begin{equation*}
\sum_{K_1, K_2, \ldots, K_s \geq 0} (K_1 + K_2 + \ldots + K_s ) l_{K_1, K_2, \ldots, K_s} = N. 
\end{equation*}
To see this, we note that a gap of the form $K_1 d_1 + K_2 d_2+ \ldots + K_s d_s$ comprises $K_1 + K_2+ \ldots + K_s $ points of the original set $\lbrace z_1, \ldots, z_{N} \rbrace$. \\

If all $l_{K_1,K_2, \ldots, K_s}$, with $ K_1 + K_2 + \ldots +K_s \leq su$, where $u=2/\gamma$, would be smaller than $\epsilon_0N$, $\epsilon_0:= \frac{\gamma^{s+1}}{2^{s+1}s^s}$, we had
\begin{align*}
N \geq &  \sum_{ \substack{K_1, K_2, \ldots, K_s \geq 0 ,\\ K_1 + K_2 + \ldots +K_s \geq su+1} } (K_1 + K_2 + \ldots + K_s)  l_{K_1, K_2, \ldots, K_s} \\
\geq & (su+1)  \sum_{ \substack{K_1, K_2, \ldots, K_s \geq 0 ,\\ K_1 + K_2 + \ldots +K_s \geq su+1} } l_{K_1, K_2, \ldots, K_s} \\
\geq & \left(\gamma N -1 - (su)^s \epsilon_0 N \right) su= \left( \gamma N -1 - \frac{2^s s^s}{\gamma^s}\frac{\gamma^{s+1}}{2^{s+1}s^s}N \right) \frac{2s}{\gamma} = sN - \frac{2s}{\gamma},
\end{align*}
which is a contradiction. We can conclude: \\ \\
\textit{Result $I$:}\\
There exist indices $K_1, K_2, \ldots, K_s$, with $K_1 +K_2 +\ldots +K_s \leq su$, such that $l_{K_1, K_2, \ldots, K_s} \geq \epsilon_0 N$, i.e., we have a gap of the form $K_1 d_1 + K_2 d_2 + \ldots + K_s d_s $ amongst the set of points $x_{j_1}, \ldots, x_{j_M}$ which appears at least $\epsilon_0 N$ many times. \\ \\
\textit{Result $II$:}\\
If instead of (\ref{eq:eq11}) we had the slightly different condition 
\begin{equation}\label{eq:eq12}
\sum_{K_1, K_2, \ldots, K_{s-1} \geq 0} \sum_{K_s \geq 1}  l_{K_1, K_2, \ldots, K_s} \geq \ \tau N 
\end{equation}
for some $\tau > 0$, then we can show in complete analogy to \textit{Result $I$}, that there exist indices $K_1, K_2, \ldots, K_{s-1} \geq 0$ and  $K_s \geq 1$, with $K_1 +K_2 +\ldots +K_s \leq su$, such that $l_{K_1, K_2, \ldots, K_s} \geq \epsilon_1 N$, where $\epsilon_1 = \epsilon_1(\tau) > 0$. Of course, the role of $K_s$ in the sum (\ref{eq:eq12}) could be interchanged with any other $K_j$, for $j=1, \ldots, s-1$. \\ 
 
We consider now the sequences $(d_j^{(i)})_{i \in \NN}$, for $j=1, \ldots, s$ and distinguish the following cases. First, note that there do not exist subsequences $(i_k)_{k \in \NN}$, such that $N_{i_k} d_j^{(i_k)} \to 0$ or $N_{i_k} d_j^{(i_k)} \to \infty$, as $k \to \infty$, for all $j=1, \ldots, s$. 
\begin{enumerate}
\item Assume that there exists a subsequence $(i_k)$ (w.\ l.\ o.\ g., we assume in the sequel $(i_k)=(i)$), such that $N_i d_j^{(i)}$ is bounded from below and above, for $j=1, \ldots, w_1$, and $N_i d_j^{(i)} \to \infty$, as $i \to \infty$, for $j= w_1 +1, \ldots, s$. We know due to \textit{Result $I$} that for each $i$ there exist $K_j^{(i)}$, $j=1, \ldots, s$, such that the gap of length $g_i = K_1^{(i)} d_1^{(i)} + \ldots +  K_s^{(i)} d_s^{(i)}$ appears at least $\epsilon_0 N_i$ times. Consequently $g_i \epsilon_0 N_i \leq 1$. Therefore, 
\begin{equation*}
K_{w_1 +1}^{(i)} d_{w_1 +1}^{(i)} + \ldots + K_s^{(i)} d_{s}^{(i)} \leq g_i \leq \frac{1}{\epsilon_0 N_i},
\end{equation*} 
and finally 
\begin{equation*}
K_{w_1 +1}^{(i)} + \ldots + K_s^{(i)} \leq \frac{1}{\epsilon_0 N_i  d_{w_1+1}^{(i)}} \to 0,\text{ as } i \to \infty, 
\end{equation*}
which implies that all $K_j^{(i)} =0$, $j=w_1 +1, \ldots, s$, for $i$ large enough. I.e., for a sufficiently large $i$ there exists a gap of the form  $K_1^{(i)} d_1^{(i)} + \ldots + K_{w_1}^{(i)} d_{w_1}^{(i)}$, which appears at least $\epsilon_0 N_i$ times. Due to this consideration, we can conclude that $g_i N_i$ is bounded from above and below and therefore the strategies of \textit{Proof Part 1} of Proposition 1 are applicable. \\  
\item If $N_i d_j^{(i)}$ is bounded from above and below for all $j=1, \ldots, s$, then, due to \textit{Result $I$}, we can immediately apply \textit{Proof Part 1} of Proposition 1. \\
\item Assume now
\begin{align*}
&\lim_{i \to \infty} N_i d_j^{(i)} = 0, \text{ for } j=1, \ldots, w_1 \text{ and } \\
&\lim_{i \to \infty} N_i d_j^{(i)} = \infty, \text{ for } j=w_1+1, \ldots, s.
\end{align*}
This implies that for all $\epsilon > 0$ the number of gaps amongst the set of points  $x_{j_1}, \ldots, x_{j_M}$ of length $\geq \epsilon \frac{1}{N_i}$ is, as $i \to \infty$, of order $o(N_i)$ (otherwise more than $N_i$ points of $\lbrace z_1, \ldots, z_{N_i} \rbrace$ have to be removed to obtain $\Omega(N_i)$ many intervals of length $\geq  \epsilon \frac{1}{N_i}$). Thus, it is admissible to apply \textit{Proof Part 2} of Proposition 1.  
\item Assume $\lim_{i \to \infty} N_i d_j^{(i)} = 0, \text{ for } j=1, \ldots, w_1$ and $N_i d_j^{(i)}$ is bounded from above and below, for $j=w_1+1, \ldots, s$. If the number of gaps with length $d_j^{(i)}$, for all $j=w_1+1, \ldots, s$, is of order $o(N_i)$, then we argue as in Case 3. Assume, w.\ l.\ o.\ g., that the number of the gaps with length $d_s^{(i)}$ is $\geq \tau N_i$, for all $i$, for some $\tau >0$. Then, condition (\ref{eq:eq12}) is satisfied and we can apply \textit{Result $II$} in order to finally proceed as in \textit{Proof Part 1} of Proposition 1. \\
\item Finally, we consider the case 
\begin{equation*}
\lim_{i \to \infty} N_i d_j^{(i)} = 0, \text{ for } j=1, \ldots, w_1 \text{ and } \lim_{i \to \infty} N_i d_j^{(i)} = \infty, \text{ for } j=w_2, \ldots, s,
\end{equation*}
and $N_i d_j^{(i)}$ is bounded from above and below for $j=w_1 +1, \ldots, w_2-1$. Combining the Cases 1, 3 and 4, allows to deduce the Theorem in this case as well. \hfill $\square$
\end{enumerate}
\section{Proof of Theorem 2}            
\begin{proof}
We investigate the pair correlations of the sequence $x_n:=\lbrace 3^n \alpha \rbrace$, $n=1,2, \ldots$, where $\alpha=0.a_1a_2\ldots$ and $a_1a_2\ldots$ is an infinite de Bruijn word. For simplicity, we consider the alphabet $A = \lbrace 0,1,2\rbrace$ and set $s=2$ and $N=3^m$ in (\ref{eq:pc1}). If the first $m$ digits of two distinct sequence elements $x_l, x_k$ match, then their distance is less than $2/N$. Since the string $a_1 \ldots a_{3^m+m-1}$, contains by the definition of a de Bruijn word, every word of length $m$ exactly once, this case cannot occur. Let us now consider words starting with blocks of the form $\underbrace{a_1 \ldots a_j 2 0 \ldots 0}_{m \text{ digits }}$ and $\underbrace{a_1 \ldots a_j 1 2 \ldots 2}_{m \text{ digits } }$. The resulting pairs of sequence elements, denoted by $x_k$ and $x_l$, have a distance less than $2/N$, as 
\begin{equation*}
\| x_k -x_l \| \leq \frac{1}{3^m} + \frac{2}{3^{m+1}} + \frac{2}{3^{m+2}} + \ldots \leq \frac{2}{3^{m}} = \frac{2}{N}.
\end{equation*}
Here, we obtain at most (summing over all possible positions for $j$) 
\begin{equation*}
2 \sum_{j=1}^{m-1} 3^{j}  = -3 + 3^m
\end{equation*}     
pairs with distance less than $2/N$. To see this, we mention that for a fixed $j$ there are $3^j$ possible choices for the word $a_1 \ldots a_j$. As we are considering ordered pairs the factor $2$ in front of the sum is necessary. Another possibility which yields pairs with a prescribed distance, is to consider words which have starting blocks of the form $\underbrace{a_1 \ldots a_j 1 0 \ldots 0}_{m \text{ digits }} $ and $\underbrace{a_1 \ldots a_j 0 2 \ldots 2}_{m \text{ digits } }$. Again, also in this case, we obtain $3^m -3$ ordered pairs with distance less than $2/N$. Finally, consider words of the form $\underbrace{a_1 \ldots a_{m-1} 2}c_1c_2c_3 \ldots $ and $\underbrace{a_1 \ldots a_{m-1} 0}d_1d_2d_3\ldots$, with $d_1d_2d_3 \ldots > c_1c_2c_3 \ldots$. Here, we get $3^{m-1}2$ possible pairs. All these pairs have a distance less than $2/3^m$, due to the carry caused by the requirement $d_1d_2d_3 \ldots > c_1c_2c_3 \ldots$. \\ \\
Above cases show that there are at most $2(3^m + 3^{m-1} -3)$ pairs having a distance less than $2/N$. We therefore have for the pair correlation statistics 
\begin{equation*}
\lim_{N \to \infty} F_N(2) \leq 8/3 < 2s.
\end{equation*}
Note that the amount of pairs with distance $> 1-2/N$ is negligible; the only possible structure for two words achieving this distance is $a_1 \ldots a_m = 2 \ldots 2$ and $\tilde{a}_1 \ldots \tilde{a}_m = 0 \ldots 0$.
\end{proof}
\section{Proof of Theorem 3}
\begin{proof}
We define the sequence $x_n := \lbrace 2^n \alpha_{2,3} \rbrace$, $n=0, 1, 2, \ldots $. We have 
\begin{equation*}
\lbrace 2^n \alpha_{2,3} \rbrace = \Bigg\lbrace \sum_{m=1}^{\lfloor \log_3 n \rfloor} \frac{2^{n-3^m} \mod 3^m }{3^m} \Bigg\rbrace + \sum_{m= \lfloor \log_3 n \rfloor +1}^{\infty} \frac{2^{n-3^m}}{3^m}.
\end{equation*}
For the first term of the above expression, it is known that it can be expressed by the recursion $z_0=0$, and for $n \geq 1$, $z_n = \lbrace 2z_{n-1}+r_n \rbrace$, where $r_n=1/n$ if $n=3^k$ for some integer $k$, and zero otherwise. Moreover, it can be proven that if $n < 3^{q+1}$, for some integer $q$, then $z_n$ is a multiple of $1/3^q$ and appears exactly three times amongst the elements $z_0, \ldots, z_{3^{q+1}-1}$. To be more precise, the sequence elements $z_n$ (appearing three times) have the form $j/3^{\tilde{q}}$, where $\tilde{q} \leq q$ and $\gcd(j,3)=1$. For details on this description, see, e.g., \cite{not16,not18}. We choose $N=2^w$, for some integer $w$ and $s=1$ in the definition of the pair correlation statistics. Let the integer $l$ be chosen in such a way that $3^l < N=2^w < 3^{l+1}$. First, we give some additional information on the distribution of the sequence $z_n$ in the unit interval. \\
 
Assume, we are given an interval $[c,d)$, where $c=0.y_1y_2y_3 \ldots y_w$, with $y_i \in \lbrace 0,1 \rbrace$ for $i=1, \ldots, w$, and $d$ is the next largest binary fraction of length $w$, such that $d-c=2^{-w}$. It follows that the sequence elements $z_n$ (for $n < N$), are multiples of $1/3^l$. Further, note that the interval $[c,d)$ has length $1/2^w=1/N$. Therefore, it contains at most $\lfloor 3^l 2^{-w} \rfloor + 1 = 1$ integer multiples of $1/3^l$, i.e., either 3 elements of the sequence $z_n$ are contained in the interval $[c,d)$ or no element at all. These considerations suggest that two sequence elements $\lbrace 2^{n_1} \alpha_{2,3} \rbrace$, $\lbrace 2^{n_2} \alpha_{2,3} \rbrace$ with $z_{n_1} = z_{n_2}$ have small distance (at least most of them), i.e., distance less than $1/N$ and those with $z_{n_1} \neq z_{n_2}$ a distance larger than $1/N$. In the sequel, if we use the expression "most of the elements", we mean all such elements except $o(N)$ many.  \\

Consider now the difference between two sequence elements $\lbrace 2^{n_1} \alpha_{2,3} \rbrace$, \\ $\lbrace 2^{n_2} \alpha_{2,3} \rbrace$, with $n_1 \neq n_2 < N$. We have 
\begin{equation*}
\lbrace 2^{n_1} \alpha_{2,3} \rbrace - \lbrace 2^{n_2} \alpha_{2,3} \rbrace = z_{n_1} - z_{n_2} + \sum_{m= \lfloor \log_3 n_1 \rfloor +1}^{\infty} \frac{2^{n_1-3^m}}{3^m} - \sum_{m= \lfloor \log_3 n_2 \rfloor +1}^{\infty} \frac{2^{n_2-3^m}}{3^m}.
\end{equation*}  
First, we demonstrate that indeed most of the sequence elements with a common value for $z_n$ have a small distance. Since by a basic property of the Eulerian totient function, we know that there are indices 
\begin{equation*}
n_i \leq 3^l-2(3^{l-1} -3^{l-2}) = 3^l-3^{l-2}2,
\end{equation*}
such that $z_{n_i} = z_{n_i + 3^{l-1} -3^{l-2}} = z_{n_i + 2(3^{l-1} -3^{l-2})}$ (recall that the nominators of the elements of the sequence $z_n$ are relatively prime to the denominators). The difference between the sequence elements $x_{n_i}$ and $x_{n_i + 3^{l-1} -3^{l-2}}$ can therefore be expressed by (similar for the other two differences)
\begin{equation*}
| \sum_{m= \underbrace{\lfloor \log_3 n_i \rfloor +1}_{=l}}^{\infty} \frac{2^{n_i-3^m}}{3^m} \left(1-2^{{3^{l-2}2}} \right)  |,
\end{equation*}
which is less than $1/N$ for most values of $n_i$. Therefore, we roughly estimate the number of pairs $(x_{n_i}, x_{n_j})$ (for indices $n_i, n_j  \leq 3^l$) with distance $< 1/N$ from above by $3^{l-1}6= 3^l 2 < 2^{w+1}$. To see this, we note that we have, for indices $n_i, n_j, n_k \leq 3^l$, at most $3^{l-1}$ triples $(x_{n_i}, x_{n_j}, x_{n_k})$, where each such triple yields $6$ ordered pairs with distance $<1/N$. If the sequence elements $x_n$, $3^l < n < 2^w $, have distance $> 1/N$ from each other, we immediately get that the pair correlation statistics satisfies $F_N(1)< 2^{w+1}/2^w = 2$ (note that for now we have assumed that two distinct sequence elements $x_{n_1}$ and $x_{n_2}$, with $z_{n_1} \neq z_{n_2}$ and $0 \leq n_1 \neq n_2  < N$, do have a distance $> 1/N$ (at least, in some sense, most of them)). Let us now consider the case that we get additional $3^{l-1} 6 $ pairs due to the sequence elements with indices $>3^l$. This means that $3^l2 < 2^w$ and consequently 
\begin{equation*}
F_N(1) \leq \frac{3^{l}2 + 3^{l-1}6}{2^w} = \frac{3^l( 2 + 2)}{2^w} < 2.
\end{equation*}

It remains to show that two distinct sequence elements $x_{n_1}$ and $x_{n_2}$, with $z_{n_1} \neq z_{n_2}$ and $0 \leq n_1 \neq n_2  < N$, do have a distance $> 1/N$ (at least, in some sense, most of them). Succeeding in showing this fact, would allow us to conclude that the Stoneham number is not Poissonian. We know that 
\begin{equation*}
| z_{n_1} - z_{n_2} | = \Big| \frac{c_1}{3^{l_1}} - \frac{c_2}{3^{l_2}} \Big| = \Big| \frac{1}{3^{l_2}}(\tilde{c}_1 - c_2) \Big| \geq 1/3^l > 1/N, 
\end{equation*}
where $c_1$ is relatively prime to $3^{l_1}$ and $c_2$ to $3^{l_2}$. In the sequel, we assume $z_{n_1} > z_{n_2}$. If 
\begin{equation*}
\sum_{m= \lfloor \log_3 n_1 \rfloor +1}^{\infty} \frac{2^{n_1-3^m}}{3^m} - \sum_{m= \lfloor \log_3 n_2 \rfloor +1}^{\infty} \frac{2^{n_2-3^m}}{3^m} \geq  0,
\end{equation*}
then we are done.  \\ 
In case this expression is negative, we use the following argument. The first summand of 
\begin{equation*}
\sum_{m= \lfloor \log_3 n_2 \rfloor +1}^{\infty} \frac{2^{n_2-3^m}}{3^m} = \sum_{m= l_2+1}^{\infty} \frac{2^{n_2-3^m}}{3^m}
\end{equation*}
is 
\begin{equation*}
\frac{1}{3^{l_2 +1} 2^{3^{l_2+1}-n_2}}.
\end{equation*}
If $\tilde{c}_1 - c_2 \geq 2$, $3^{l_2+1}-n_2 \geq 1$, then we have $| x_{n_1} - x_{n_2} | > 1/2^w$. It remains to check the case that $\tilde{c}_1 - c_2 = 1$ and $3^{l_2+1}-n_2 =1$. This might yield pairs with distance less than $1/N$, but at most $2l = \mathcal{O}(\ln N)$ pairs, i.e., a negligible number if we consider the relative amount. 
\end{proof}
  
\textbf{Author's Addresses:} \\ 
Gerhard Larcher and Wolfgang Stockinger, Institut f\"ur Finanzmathematik und Angewandte Zahlentheorie, Johannes Kepler Universit\"at Linz, Altenbergerstra{\ss}e 69, A-4040 Linz, Austria. \\ \\ 
Email: gerhard.larcher(at)jku.at, wolfgang.stockinger(at)jku.at      
\end{document}